\documentclass[11pt]{article}
\usepackage[utf8]{inputenc}

\usepackage{mathrsfs}
\usepackage{rotating}
\usepackage{graphicx}
\usepackage{color}
\usepackage{latexsym}
\usepackage{amsfonts}
\usepackage{amsmath}
\usepackage{amssymb}
\usepackage{psfrag}
\usepackage{mathtools,bbm}
\usepackage{algpseudocode}
\usepackage{caption}
\usepackage[ruled]{algorithm}
\usepackage{amsthm}
\usepackage{multicol}
\usepackage{url}

\usepackage[]{todonotes}

\newtheorem{proposition}{Proposition}

\newtheorem{theorem}{Theorem}

\newtheorem{lemma}{Lemma}

\usepackage{mathrsfs}

 \setbox\strutbox=\hbox{\vrule height7pt depth2pt width0pt}

\def\square{{\setbox0=\hbox{X}\hbox to \ht0{\vrule\hss\vbox to \ht0{
  \hrule width \ht0\vfil\hrule width \ht0}\vrule}}}

\newcommand{\R}{{\mathbb R}}
\newcommand{\Z}{{\mathbb Z}}

\newcommand{\krn}[1]{\text{ker}(#1)}

\newcommand{\ran}[1]{\text{ran}(#1)}
\newcommand{\argmin}[1]{\text{argmin}\,#1}

\begin{document}

\vskip 2cm \begin{center}{\huge An Active Set Algorithm for Robust Combinatorial Optimization Based on Separation Oracles }\end{center}
\par\bigskip
\centerline{\Large C. Buchheim$^\dag$ and M. De Santis$^*$}
\par\bigskip\bigskip

\centerline{$^{\dag}$Fakult\"at f\"ur Mathematik } \centerline{TU
    Dortmund } \centerline{Vogelpothsweg 87 - 44227 Dortmund - Germany}
\par\medskip
\centerline{ $^{*}$ Dipartimento di Ingegneria Informatica, Automatica e Gestionale}
\centerline{Sapienza, Universit\`a  di Roma}\centerline{Via Ariosto, 25 - 00185 Roma -Italy}
\par\medskip
\centerline{e-mail (Buchheim): christoph.buchheim@math.tu-dortmund.de}
\centerline{e-mail (De Santis): marianna.desantis@uniroma1.it}

\par\bigskip\noindent {\small \centerline{\bf Abstract} 
   We address combinatorial optimization problems with uncertain
  coefficients varying over ellipsoidal uncertainty sets. The robust
  counterpart of such a problem can be rewritten as a second-oder cone
  program (SOCP) with integrality constraints. We propose a
  branch-and-bound algorithm where dual bounds are computed by means
  of an active set algorithm. The latter is applied to the Lagrangian
  dual of the continuous relaxation, where the feasible set of the
  combinatorial problem is supposed to be given by a separation
  oracle.  The method benefits from the closed form solution of the
  active set subproblems and from a smart update of pseudo-inverse
  matrices. We present numerical experiments on randomly generated
  instances and on instances from different combinatorial problems,
  including the shortest path and the traveling salesman problem,
  showing that our new algorithm consistently outperforms the
  state-of-the art mixed-integer SOCP solver of Gurobi.
  
\par\bigskip\noindent
{\bf Keywords.} Robust Optimization, Active Set Methods, SOCP

\pagestyle{plain} \setcounter{page}{1}

\section{Introduction}

We address combinatorial optimization problems given in the general
form
\begin{equation}\label{CP}\tag{\textnormal{CP}}
\min_{x\in P\cap\Z^n}\; c^\top x
\end{equation}
where $P\subseteq \R^n$ is a compact convex set, say $P\subseteq
[l,u]$ with $l,u\in\R^n$, and the objective function vector~$c\in
\R^n$ is assumed to be uncertain.  This setting appears in many
applications where the feasible set is certain, but the objective
function coefficients may have to be estimated or result from imprecise
measurements. As an example, when searching for a shortest path in a
road network, the topology of the network is usually considered fixed,
but the travel times may vary depending on the traffic
conditions.

A classical way of dealing with uncertain optimization problems is the
strictly robust optimization approach, introduced in~\cite{bental99}
for linear programming and in~\cite{bental98} for general convex
programming; we also refer the reader to the book by Ben-Tal and
Nemirovski~\cite{bental2001}.  In strictly robust optimization, we
look for a worst-case solution, where the uncertain parameter~$c$ is
assumed to belong to a bounded set~$U\subseteq \R^{n}$, called the
\emph{uncertainty set}, and the goal of the robust counterpart is to
compute the solution of the following min-max problem:
\begin{equation}\label{RP}\tag{\textnormal{RP}}
\min_{x\in P\cap \Z^n}\;\max_{c\in U}\; c^\top x 
\end{equation}
A natural choice in this approach are ellipsoidal uncertainty sets,
defined as
\[U = \{c\in \R^n \mid (c-{\bar c})^\top M (c-{\bar c})\le 1\},\]
where $M\in\R^{n\times n}$ is a symmetric positive definite matrix
and~${\bar c}\in\R^n$ is the center of the ellipsoid. Assuming that
the uncertain vector~$c$ in~\eqref{CP}, considered as a random
variable, follows a normal distribution, we can interpret the
ellipsoid~$U$ as a confidence set of~$c$; in this case,~$M$ is the
inverse covariance matrix of~$c$ and ${\bar c}$ is its expected
value. Unfortunately, for ellipsoidal
uncertainty sets, the robust counterpart~\eqref{RP} is usually much
harder to solve than the original problem~\eqref{CP}: it is known
that Problem~\eqref{RP} is NP-hard in this case for the
shortest path problem, for the minimum spanning tree problem, and for
the assignment problem~\cite{kouvelis} as well as for the
unconstrained binary optimization problem~\cite{survey}.

Even in the case of a diagonal
matrix~$M$, i.e., when ignoring correlations and only taking variances
into account, no polynomial time algorithm for the robust shortest
path problem is known. There exists however an FPTAS for the diagonal
case whenever the underlying problem~\eqref{CP} admits an
FPTAS~\cite{nikolova}, and polynomial time algorithms for the minimum
spanning tree problem and the unconstrained binary problem have been
devised for the diagonal case.

For general ellipsoids~$U$, most exact solution approaches
for~\eqref{RP} are
based on solving SOCPs. In fact, it is easy to see
that the optimal solution of the inner maximization problem \[\max_{c\in U}\; c^\top x\]
for fixed~$x$ is given by 
\[{{\bar c}}^\top x + \sqrt{x^\top M^{-1} x}.\]
Therefore, Problem~\eqref{RP} is equivalent to the integer
non-linear problem
\begin{equation}\label{ellvILP}\tag{\textnormal{P}}
  \begin{array}{l l}
    \min & f (x) = c^{\top}x+\sqrt{x^{\top}Qx}\\
    \textnormal{ s.t. }& x \in P\cap\Z^n
  \end{array}
\end{equation}
where $Q\in\R^{n\times n}$ is the symmetric and positive definite
inverse of $M$ and we replace $\bar c$ by $c$ for ease of notation. Note
that, when
addressing so called value-at-risk models
\[
\begin{array}{ll}
  \min &  z \\
  \textnormal{ s.t.} & \text{Pr}(c^\top x\ge z)\le \varepsilon\\
  & x\in P\cap\Z^n\;,
\end{array}
\]
we arrive at essentially the same
formulation~\eqref{ellvILP}, assuming normally distributed coefficients again; see,
e.g.,~\cite{nikolova}.

In the following, we assume that the convex set~$P$ is given by a
separation algorithm, i.e., an algorithm that decides whether a given
point~$\bar x\in\R^n$ belongs to $P$ or not, and, in the negative
case, provides an inequality~$a^\top x\le b$ valid for~$P$ but
violated by~$\bar x$. Even in cases where the underlying
problem~\eqref{CP} is tractable, the polytope~$\text{conv}(P\cap
\Z^n)$ may have an exponential number of facets, so that a full linear
description cannot be used efficiently. This is true, e.g., for the
standard formulation of the spanning tree problem. However, we do not
require that a complete linear description of $\text{conv}(P\cap
\Z^n)$ be known; it suffices to have an integer linear description,
i.e., we allow~$P\neq\text{conv}(P\cap\Z^n)$. In particular, our
approach can also be applied when the underlying problem is
NP-hard, e.g., when~\eqref{CP} models the traveling salesman problem.

As soon as $P$ is given explicitly by linear constraints $Ax\le b$
with~$A\in\R^{m\times n}$ and~$b\in\R^m$, the continuous relaxation of
Problem~\eqref{ellvILP}
reduces to an SOCP of the form
\begin{equation}\label{ellLP}\tag{\textnormal{R2}}
      \begin{array}{l l}
        \min & c^{\top}x+\sqrt{x^{\top}Qx}\\
        \textnormal{ s.t. }& Ax \leq b\\
        & x\in\R^n\;.
      \end{array}
\end{equation}
Such SOCPs can be solved efficiently using interior point
algorithms~\cite{NN1993} and popular solvers for SOCPs such as
SeDuMi~\cite{S98guide} or MOSEK~\cite{mosek} are based on interior
point methods.  However, in our branch-and-bound algorithm, we need to
address a sequence of related SOCPs. Compared with interior point
methods, active set methods have the advantage to allow warmstarting
rules.

For this reason, in order to solve the SOCP relaxations of
Problem~\eqref{RP}, we devised the active set
algorithm~\texttt{EllAS}. It is applied to the Lagrangian dual
of~\eqref{ellLP} and exploits the fact that the active set subproblems
can be solved by closed form expressions.  For this, the main
ingredient is the pseudo-inverse of~$AQ^{-\frac 12}$. Since the
matrix~$A$ is updated in each iteration of the active set method, an
incremental update of the pseudo-inverse is crucial for the running
time of~\texttt{EllAS}. Altogether, we can achieve a running time of
$O(n^2)$ per iteration. Combined with an intelligent embedding into
the branch-and-bound scheme, we obtain an algorithm that consistently
outperforms the MISOCP solver of Gurobi~7.5.1, where the latter is either
applied to a full linear description of~$P$ or, in case a compact linear
description does not exist, uses
the same separation oracle as~\texttt{EllAS}.

The rest of the paper is organized as follows: the Lagrangian dual
of~\eqref{RP} is derived in Section~\ref{sec:dualProb}. The
closed-form solution of the resulting active set subproblems is
developed in Section~\ref{sec:solSubp}. The active set
algorithm~\texttt{EllAS} is detailed and analyzed in
Sections~\ref{sec:Alg} and~\ref{sec:conv}. In Section~\ref{sec:bb}, we
discuss how to embed~\texttt{EllAS} into a branch-and-bound
algorithm. Numerical results for random integer instances as well as
instances of different combinatorial optimization problems are
reported in Section~\ref{sec:num}. Section \ref{sec:conc} concludes.

\section{Dual problem}\label{sec:dualProb}

The algorithm we propose for solving Problem~\eqref{RP} uses the
Lagrangian dual of relaxations of the form~\eqref{ellLP}.  Let
$\mathscr{L}(x,\lambda) : \R^n\times \R^m \rightarrow \R$ be the
Lagrangian function associated to~\eqref{ellLP}:
\begin{equation*}
  \mathscr{L}(x,\lambda) = c^\top x+\sqrt{x^{\top}Q x} +
  \lambda^\top (Ax-b)\;.
\end{equation*}
The Lagrangian dual of Problem~\eqref{ellLP} is then
\begin{equation}\label{lagr_dual}
  \max_{\lambda\in\R^m_+} ~ \inf_{x\in\R^n} ~ \mathscr{L}(x,\lambda)\;.
\end{equation}
After applying the bijective transformation $z=Q^{\frac 12}x$, the
inner minimization problem of~\eqref{lagr_dual} becomes
\begin{equation*}
  -b^\top \lambda + \inf_{z\in\R^n} \big(Q^{-\frac
    12}(c+A^\top\lambda)\big)^\top z + \|z\|
\end{equation*}
for fixed~$\lambda\in\R^m_+$. It is easy to see that 
\[
\inf_{z\in\R^n} \big(Q^{-\frac 12}(c+A^\top\lambda)\big)^\top z + \|z\| =
\min_{z\in\R^n} \big(Q^{-\frac
    12}(c+A^\top\lambda)\big)^\top z + \|z\| = 0
\]
if $\|Q^{-\frac 12}(c +
A^\top\lambda)\|\leq 1$ and $-\infty$ otherwise. Therefore,
Problem~\eqref{lagr_dual} reduces to
\begin{equation}\label{lagr_dual_red}\tag{\textnormal{D}}
  \begin{array}{l l}
    \max & -b^\top\lambda\\
    \textnormal{ s.t. } & (c + A^\top\lambda)^\top Q^{-1} (c+A^\top\lambda)\leq 1\\
    &\lambda\geq 0\;.
  \end{array}
\end{equation}

\begin{theorem}\label{duality}
  For the primal-dual pair of optimization problems~\eqref{ellLP}
  and~\eqref{lagr_dual_red}, strong duality holds as soon as one of
  the two problems is feasible. Moreover, if one of the problems admits
  an optimal solution, the same holds for the other problem.
\end{theorem}
\begin{proof}
  This follows from the convexity of~\eqref{ellLP} and from the fact
  that all constraints in~\eqref{ellLP} are affine linear.
  \qed
\end{proof}
In order to solve Problem~\eqref{ellLP}, we have devised the dual
active set algorithm \texttt{EllAS} detailed in
Section~\ref{sec:Alg}. Along its iterations, \texttt{EllAS} produces
dual feasible solutions of Problem~\eqref{lagr_dual_red}, converging
to a KKT point of Problem~\eqref{ellLP} and therefore producing also a
primal optimal solution when terminating.

\section{Solving the Active Set Subproblem}\label{sec:solSubp}

At every iteration, the active set algorithm \texttt{EllAS} presented
in the subsequent sections fixes
certain dual variables to zero while leaving unconstrained 
the remaining variables. 
In the primal problem, this
corresponds to choosing a set of valid linear constraints~$Ax\le b$
for~$P$ and replacing inequalities by equations.
We thus need to solve primal-dual pairs of problems of the following type:

\begin{align}
  \min~~ & f(x) = c^{\top}x+\sqrt{x^{\top}Qx}\nonumber\\
  \textnormal{s.t.}~~ & \hat Ax = \hat b\label{P-AS}\tag{\textnormal{P--AS}}\\
  & x\in\R^n\;\nonumber\\[2ex]
  \max~ & -\hat b^\top\lambda\nonumber\\
  \textnormal{s.t.}~~ & (c + \hat A^\top\lambda)^\top Q^{-1} (c+\hat A^\top\lambda)\leq 1  \label{D-AS}\tag{\textnormal{D--AS}}\\
  & \lambda\in\R^{\hat m}\nonumber
\end{align}
where $\hat A\in \R^{\hat m\times n}$, $b\in \R^{\hat m}$. 
For the efficiency of our algorithm, it is crucial that this pair of
problems can be solved in closed form. For this, the pseudo-inverse
$(\hat AQ^{-\frac 12})^+$ of $\hat AQ^{-\frac 12}$ will play an
important role. It can be used to compute orthogonal projections onto
the kernel and onto the range of~$Q^{-\frac 12}\hat A^\top$ as
follows: we have
\begin{equation}\label{eq:projker}
  \text{proj}_{\krn{Q^{-\frac 12}\hat A^\top}}(y) = y-\hat AQ^{-\frac 12}(\hat AQ^{-\frac 12})^+y
\end{equation}
and
\begin{equation}\label{eq:projran}
  \text{proj}_{\ran{Q^{-\frac 12}\hat A^\top}}(y) = (\hat AQ^{-\frac 12})^+\hat AQ^{-\frac 12}y\;,
\end{equation}
see e.g.~\cite{Me2000}. We later explain how to update the pseudo-inverse incrementally
instead of computing it from scratch in every iteration, which would
take $O(n^3)$ time; see Section~\ref{sec:runningtime}.

In the following, we assume that the dual problem~\eqref{D-AS} admits
a feasible solution; this will be guaranteed in every iteration of our
algorithm; see Lemma~\ref{feasD-AS} below.

\subsection{Dual Unbounded Case}

If $\hat b\not\in \ran{\hat A}$, or equivalently, if~$\hat b$ is not
orthogonal to~$\krn{\hat A^\top}=\krn{Q^{-\frac 12}\hat A^\top}$, then
the dual problem~\eqref{D-AS} is unbounded, and the corresponding
primal problem~\eqref{P-AS} is infeasible. When this case occurs,
\texttt{EllAS} uses an unbounded direction of~\eqref{D-AS} to
continue. The set of unbounded directions of~\eqref{D-AS}
is~$\krn{Q^{-\frac 12}\hat A^\top}$. Consequently, the unbounded direction with
steepest ascent can be obtained by projecting the gradient of the
objective function~$-\hat b$
to~$\krn{Q^{-\frac 12}\hat A^\top}$. According to~\eqref{eq:projker}, this projection is
\[\text{proj}_{\krn{Q^{-\frac 12}\hat A^\top}}(-\hat b) = (\hat AQ^{-\frac
  12})(\hat AQ^{-\frac 12})^+\hat b-\hat b\;.\]

\subsection{Bounded Case}\label{sec:bounded}

If $\hat b\in \ran A$, we first consider the special case $\hat
b=0$. As we assume~\eqref{D-AS} to be feasible, its optimum value is thus
$0$. Therefore, the corresponding primal problem~\eqref{P-AS} admits
$x^* = 0$ as optimal solution.
In the following, we may thus assume $\hat b \neq 0$.
The feasible set of problem~\eqref{D-AS} consists of all~$\lambda\in\R^{\hat m}$ such that
$$||Q^{-\frac 1 2} (c+\hat A^\top\lambda)||\le 1\;,$$ i.e., such that the
image of~$\lambda$ under $-Q^{-\frac 1 2}\hat A^\top$ belongs to the
ball~$B_1(Q^{-\frac 1 2}c)$. Consider the orthogonal projection of
$Q^{-\frac 1 2}c$ to the subspace~$\ran{Q^{-\frac 1 2}\hat A^\top}$,
which by~\eqref{eq:projran} is
$$p:=\text{proj}_{\ran{Q^{-\frac 12}\hat A^\top}}(Q^{-\frac 1 2}c)=(Q^{-\frac 1 2}\hat A^\top)(Q^{-\frac 1 2}\hat A^\top)^+ Q^{-\frac 1 2}c\;.$$
If~$||p-Q^{-\frac 12}c||>1$, then the intersection $B_1(Q^{-\frac 1 2}c)\cap \ran{Q^{-\frac 1 2}\hat A^\top}$ is empty, 
so that Problem~\eqref{D-AS} is infeasible, contradicting our assumption.
Hence, we have that this intersection is a ball with center~$p$ and
radius $$r:=\sqrt{1-||p-Q^{-\frac 12}c||^2}$$
and~$\lambda\in\R^{\hat m}$ is feasible for~\eqref{D-AS} if and
only if~$-Q^{-\frac 12}\hat A^\top\lambda\in B_r(p)$.
Since $\hat b\in \ran {\hat AQ^{-\frac 12}}$, we have~$(\hat AQ^{-\frac 12})(\hat AQ^{-\frac 12})^+\hat b=\hat b$. This
allows us to rewrite the objective function~$-\hat b^\top\lambda$
of~\eqref{D-AS} in terms of~$Q^{-\frac 12}\hat A^\top\lambda$
only, as $$-\hat b^\top\lambda=-\hat b^\top (Q^{-\frac 12}\hat A^\top)^+ (Q^{-\frac 12}\hat A^\top)\lambda\;.$$
We can thus first compute the optimal solution~$v^*\in\ran{Q^{-\frac 12}\hat A^\top}$ of
\begin{equation*}
  \begin{array}{l l}
    \max & \; \hat b^\top (Q^{-\frac 12}\hat A^\top)^+ v\\
    \textnormal{ s.t.} & \; ||Q^{-\frac 12}c-v||\leq 1\;,
  \end{array}
\end{equation*}
which is unique since~$\hat b\neq 0$, and then solve~$v^*=-(Q^{-\frac
  12}\hat A^\top)\lambda$. We obtain
\begin{equation}\label{eq:vstar}
  v^*=p+\frac r{||(\hat AQ^{-\frac 12})^+\hat b||}(\hat AQ^{-\frac 12})^+\hat b\;,
\end{equation}
so that we can state the following
\begin{proposition}
 Let $\hat b\in \ran A\setminus\{0\}$ and let~$v^*$ be defined as in~\eqref{eq:vstar}. 
 Then, the unique optimal solution of~\eqref{D-AS} with minimal norm is
\[\lambda^*:=-(Q^{-\frac 12}\hat A^\top)^+v^*.\] 
\end{proposition}

\noindent From $\lambda^*$, it is possible to compute an optimal
solution~$x^*$ of the primal problem~\eqref{P-AS} as explained in the following result.

\begin{theorem}\label{Theo:PrimalFromDual}
  Let $\hat b\in \ran A\setminus\{0\}$.
  Let~$\lambda^*$ be an optimal solution of~\eqref{D-AS} and~$\bar
  x:=Q^{-1}(c+\hat A^\top\lambda^*)$. 
  \begin{itemize}
  \item[(a)] If $\hat b^\top \lambda^*\neq 0$, then the unique optimal solution
  of~\eqref{P-AS} is~$x^*=\alpha\bar x$, with  
  $$\alpha:=-\frac{\hat b^\top\lambda^*}{c^\top\bar x -\sqrt{\bar x^\top
      Q\bar x}}\;.$$
  \item[(b)] Otherwise, there exists a unique $\alpha < 0$ such that $\alpha \hat A \bar x = \hat b$. 
  Then, $x^* = \alpha \bar x$ is the unique optimal solution of~\eqref{P-AS}.
  \end{itemize}
  
\end{theorem}
\begin{proof}
  Let~$(x^*,\lambda^*)$ be a primal-dual optimal
  pair, which exists by Theorem~\ref{duality}. 
  Since~$\hat b\neq 0$ and $\hat A x^* = \hat b$, it follows
  that~$x^*\neq 0$. The gradient equation yields
  $$0=\nabla_x \mathscr{L}(x^*,\lambda^*) = c+\frac{2Qx^*}{2\sqrt{(x^*)^{\top}Q (x^*)}} +\hat A^\top\lambda^*$$
  which is equivalent to
  $$\frac{Q^{\frac 12}x^*}{\|Q^{\frac 12}x^*\|}=-Q^{-\frac
    12}(c+\hat A^\top\lambda^*)$$
  and hence to
  $$x^*=\alpha Q^{-1}(c+\hat A^\top\lambda^*)=\alpha\bar x$$
  for some~$\alpha\neq 0$.
  Since $\alpha=-\|Q^{\frac 12}x^*\|$, we have~$\alpha<0$.
  By strong duality, we then obtain
  \[
  -\hat b^\top\lambda^*=c^\top x^*+\sqrt{(x^*)^\top
    Q(x^*)}=\alpha c^\top \bar x+|\alpha|\sqrt{\bar x^\top Q\bar x}
  =\alpha\big(c^\top\bar x-\sqrt{\bar x^\top Q\bar
    x}\big)\;. 
  \]
  Now if~$\hat b^\top\lambda^*\neq 0$, also the right hand side
  of this equation is non-zero, and we obtain~$\alpha$ as claimed.
  Otherwise, it still holds that there exists $\alpha<0$ such that $\alpha \bar x$ is optimal.
  In particular, $\alpha \bar x$ is primal feasible and hence $\alpha \hat A \bar x = \hat A(\alpha \bar x) = \hat b$.
  As $\hat b\neq 0$, we derive $\hat A \bar x \neq 0$, as $\alpha<0$. This in particular shows that 
  $\alpha$ is uniquely defined by $\alpha \hat A \bar x = \hat b$.
  \qed
\end{proof}
\noindent
Note that the proof (and hence the statement) for case (b) in
Theorem~\ref{Theo:PrimalFromDual} are formally applicable also in case
(a). However, in the much more relevant case~(a), we are able to
derive a closed formula for~$\alpha$ in a more direct way.

\section{The Dual Active Set Method \texttt{EllAS}}\label{sec:Alg}

As all active set methods, our algorithm \texttt{EllAS} tries to
forecast the set of constraints that are active at the optimal
solution of the primal-dual pair~\eqref{ellLP}
and~\eqref{lagr_dual_red}, adapting this forecast iteratively:
starting from a subset of primal constraints~$A^{(1)}x\le b^{(1)}$,
where $A^{(1)}\in\R^{m^{(1)}\times n}$ and $b^{(1)}\in\R^{m^{(1)}}$,
one constraint is removed or added per iteration by performing a dual
or a primal step; see Algorithm~\ref{fig:Ell_AS}. We assume that a
corresponding dual feasible solutions~$\lambda^{(1)}\ge 0$ is given when
starting the algorithm; we explain below how to obtain this initial
solution.

\begin{algorithm}
  \caption{Ellipsoidal Active SeT algorithm {\tt EllAS}}
  \label{fig:Ell_AS}
      {\bf Input:} \hspace*{0.125cm} $Q\in\R^{n\times n}$, $c\in\R^n$, $A^{(1)}\in\R^{m^{(1)}\times n}$, $b^{(1)}\in\R^{m^{(1)}}$;\\[1ex]
      \hspace*{1.175cm} $\lambda^{(1)}\geq 0$ with $(c + (A^{(1)})^\top\lambda^{(1)})^\top Q^{-1} (c+(A^{(1)})^\top\lambda^{(1)})\leq 1$;\\
      \hspace*{1.175cm} pseudo-inverse $(A^{(1)}Q^{-\frac 12})^+$\\[1ex]
      {\bf Output:} optimal solutions of~\eqref{ellLP} and~\eqref{lagr_dual_red}
      \smallskip
      \hrule
      \smallskip
      \begin{algorithmic}[1] 
        \For {$k=1,2,3,\dots$}
        \State {{\bf solve} \eqref{lagr_dual_red_ask}  and obtain optimal $\tilde \lambda^{(k)}$ with minimal norm}
        \If {problem~\eqref{lagr_dual_red_ask} is bounded and $\tilde\lambda^{(k)}\ge 0$}
        \State {{\bf set} $\lambda^{(k)} := \tilde\lambda^{(k)}$}
        \State {{\bf perform } the primal step (Algorithm~\ref{fig:Ell_AS-PS}) and {\bf update} $x^{(k)}$, $A^{(k)}$, $b^{(k)}$}
       \Else
        \State {{\bf perform } the dual step (Algorithm~\ref{fig:Ell_AS-DS}) and {\bf update} $\lambda^{(k)}$, $A^{(k)}$, $b^{(k)}$ }    
        \EndIf
        \EndFor
      \end{algorithmic}
\end{algorithm}

At every iteration $k$, in order to decide if performing the primal or the dual step, 
the dual subproblem is addressed, namely Problem~\eqref{lagr_dual_red} where only the 
subset of active constraints is taken into account.
This leads to the following problem:
\begin{equation}\label{lagr_dual_red_ask}\tag{\textnormal{D-{ASk}}}
  \begin{array}{l l}
    \max & -{b^{(k)}}^\top\lambda\\
    \textnormal{ s.t. } & (c + {A^{(k)}}^\top\lambda)^\top Q^{-1} (c+{A^{(k)}}^\top\lambda)\leq 1\\
    & \lambda\in \R^{m^{(k)}}
  \end{array}
\end{equation}
The solution of Problem~\eqref{lagr_dual_red_ask} has been explained in Section~\ref{sec:solSubp}.
Note that formally Problem~\eqref{lagr_dual_red_ask} is defined in a smaller space with respect to 
Problem~\eqref{lagr_dual_red}, but its solutions can also be considered as elements of $\R^m$ by
setting the remaining variables to zero.

In case the dual step is performed, the solution of Problem~\eqref{lagr_dual_red_ask} gives 
an ascent direction $p^{(k)}$ along which we move in order to produce a new
dual feasible point with better objective function value. 
We set \[\lambda^{(k)} = \lambda^{(k-1)} + \alpha^{(k)} p^{(k)},\]
where the steplength $\alpha^{(k)}$ is chosen to be the largest value
for which non-negativity is maintained at all entries.
Note that the feasibility with respect to the ellipsoidal constraint 
in~\eqref{lagr_dual_red}, i.e.,
\[(c + A^\top\lambda)^\top Q^{-1} (c+A^\top\lambda)\leq 1\;,\]
is guaranteed from how $p^k$ is computed, using convexity.
Therefore, $\alpha^{(k)}$ can be derived by considering 
the negative entries of $p^{(k)}$.
In order to maximize the increase of $-b^\top \lambda$, we ask $\alpha^{(k)}$ to be as large as possible
subject to maintaining non-negativity; see Steps 9--10 in Algorithm~\ref{fig:Ell_AS-DS}.

\begin{algorithm}
  \caption{Dual Step}
  \label{fig:Ell_AS-DS}
 \begin{algorithmic}[1] 
        \If {problem~\eqref{lagr_dual_red_ask} is bounded}
        \State {{\bf set} $p^{(k)} := \tilde\lambda^{(k)} - \lambda^{(k-1)}$}
        \Else
        \State {{\bf let} $p^{(k)}$ be an unbounded direction of~\eqref{lagr_dual_red_ask} with steepest ascent}
	\If {$p^{(k)}\geq 0$}
	\State {{\bf STOP:}} primal problem is infeasible 
        \EndIf
        \EndIf
	\State {{\bf choose} $j\in \argmin \{-\lambda_i^{(k-1)}/p_i^{(k)} \mid i=1,\dots,m^{(k)},~ p^{(k)}_i < 0\}$}
        \label{step:jalpha}
        \State {{\bf set} $\alpha^{(k)} := -\lambda_j^{(k-1)}/p_j^{(k)}$}
        \State {{\bf set} $\lambda^{(k)} := \lambda^{(k-1)} +\alpha^{(k)}p^{(k)}$}
        \State {{\bf compute} $(A^{(k+1)},b^{(k+1)})$ by removing row~$j$ in $(A^{(k)},b^{(k)})$}
        \State {{\bf compute} $\lambda^{(k+1)}$ by removing entry~$j$ in $\lambda^{(k)}$}
        \State {{\bf set} $m^{(k+1)}:=m^{(k)}-1$}
        \State {{\bf update} $(A^{(k+1)}Q^{-\frac 12})^+$ from $(A^{(k)}Q^{-\frac 12})^+$}
      \end{algorithmic}
\end{algorithm}

The constraint index~$j$ computed in Step~\ref{step:jalpha} of Algorithm~\ref{fig:Ell_AS-DS} corresponds
to the primal constraint that needs to be released from the active set.
The new iterate~$\lambda^{(k+1)}$ is then obtained from $\lambda^{(k)}$, by dropping the $j$-th entry.

\begin{proposition}
  The set considered in Step~\ref{step:jalpha} of
  Algorithm~\ref{fig:Ell_AS-DS} is non-empty.
\end{proposition}
\begin{proof}
  If Problem \eqref{lagr_dual_red_ask} is bounded, there is an index
  $i$ such that $\tilde\lambda^{(k)}_i<0$, since~$\tilde\lambda^{(k)}$
  is dual infeasible.  As~$\lambda^{(k-1)}\ge 0$, we derive~$p^{(k)}_i = \tilde \lambda^{(k)}_i - \lambda^{(k-1)}_i < 0$.
  If Problem \eqref{lagr_dual_red_ask} is unbounded, we explicitly
  check whether~$p^{(k)}\ge 0$ and only continue otherwise.
  \qed
\end{proof}

The primal step is performed in case the solution of Problem~\eqref{lagr_dual_red_ask} gives 
us a dual feasible solution. Starting from this dual feasible solution, we compute a corresponding 
primal solution $x^{(k)}$ according to the formula in Theorem~\ref{Theo:PrimalFromDual}.
If $x^{(k)}$ belongs to~$P$ we are done:  
we have that $(x^{(k)}, \lambda^{(k)})$ is a KKT point of Problem~\eqref{ellLP} and, by convexity of Problem~\eqref{ellLP}, 
$x^{(k)}$ is its global optimum. 
Otherwise, we compute a cutting plane violated by $x^{(k)}$
that can be considered active and will be then taken into account in defining the 
dual subproblem~\eqref{lagr_dual_red_ask} at the next iteration.
The new iterate $\lambda^{(k+1)}$ is obtained from $\lambda^{(k)}$ by adding an entry to $\lambda^{(k)}$ 
and setting this additional entry to zero.

\begin{algorithm}
  \small
  \caption{\small\tt Primal Step}
  \label{fig:Ell_AS-PS}
 \begin{algorithmic}[1] 
        \If {$-(b^{(k)})^\top \lambda^{(k)}=0$}
        \State {{\bf STOP:} $(0,\lambda^{(k)})$ is an optimal primal-dual solution}
        \Else
        \State {{\bf compute} $x^{(k)}$ from $\lambda^{(k)}$ according to Theorem~\ref{Theo:PrimalFromDual}}
        \If {$x^{(k)}\in P$}
        \State {{\bf STOP:} $(x^{(k)},\lambda^{(k)})$ is an optimal primal-dual solution}
        \Else
        \State {{\bf compute} a cutting plane~$a^\top x\le b$ violated by $x^{(k)}$} \label{Step:separation}
        \State {{\bf compute} $(A^{(k+1)},b^{(k+1)})$ by appending $(a^\top,b)$ to $(A^{(k)},b^{(k)})$} 
        \State {{\bf compute} $\lambda^{(k+1)}$ by appending zero to $\lambda^{(k)}$}
        \State {{\bf set} $m^{(k+1)}:=m^{(k)}+1$}
        \State {{\bf update} $(A^{(k+1)}Q^{-\frac 12})^+$ from $(A^{(k)}Q^{-\frac 12})^+$}
        \EndIf
        \EndIf
 \end{algorithmic}
\end{algorithm}

\begin{theorem}
 Whenever Algorithm~\texttt{EllAS} terminates, the result is correct.
\end{theorem}
\begin{proof}
  If Algorithm~\texttt{EllAS} stops at the primal step, the
  optimality of the resulting primal-dual pair follows from the
  discussion in Section~\ref{sec:solSubp}.  If
  Algorithm~\texttt{EllAS} stops at the dual step, it means that the
  ascent direction $p^{(k)}$ computed is a feasible unbounded
  direction for Problem~\eqref{lagr_dual_red}, so that
  Problem~\eqref{lagr_dual_red} is unbounded and hence
  Problem~\eqref{ellLP} is infeasible.
  \qed
\end{proof}

It remains to describe how to initialize~\texttt{EllAS}.  For this, we
use the assumption of boundedness of~$P$ and construct~$A^{(1)}$,
$b^{(1)}$, and~$\lambda^{(1)}$ as follows: for each~$i=1,\dots,n$, we
add the constraint~$x_i\le u_i$ if~$c_i<0$, with
corresponding~$\lambda_i:=-c_i$, and the constraint~$-x_i\le -l_i$
otherwise, with~$\lambda_i:=c_i$. These constraints are valid
since we assumed~$P\subseteq[l,u]$ and it is easy to check
that~$(A^{(1)})^\top\lambda^{(1)}=-c$ by construction, so
that~$\lambda^{(1)}$ is dual feasible
for~\eqref{lagr_dual_red}. Moreover, we can easily
compute~$(A^{(1)}Q^{-\frac 12})^+$ in this case, as~$A^{(1)}$ is a
diagonal matrix with~$\pm 1$ entries: this implies~$(A^{(1)}Q^{-\frac
  1 2})^+=Q^{\frac 1 2}A^{(1)}$.

\section{Analysis of the Algorithm}\label{sec:conv}

In this section, we show that Algorithm~\texttt{EllAS} converges in a
finite number of steps if cycling is avoided. Moreover, we prove that
the running time per iteration can be bounded by~$O(n^2)$, if
implemented properly.

\subsection{Convergence Analysis}

Our convergence analysis follows similar arguments to those used in~\cite{NW2006} for the analysis 
of primal active set methods for strictly convex quadratic programming problems.
In particular, as in~\cite{NW2006}, we assume that we can always take a nonzero steplength along
the ascent direction. Under this assumption we will show that Algorithm~\texttt{EllAS} does not undergo cycling, or,
in other words, this assumption prevents from having $\lambda^{(k)} = \lambda^{(l)}$ and 
$(A^{(k)}, b^{(k)}) = (A^{(l)}, b^{(l)})$ in two different iterations $k$ and $l$.
As for other active set methods, it is very unlikely in practice to encounter a zero steplength.
However, there are techniques to avoid cycling even theoretically, such as perturbation or 
lexicographic pivoting rules in Step~\ref{step:jalpha} of Algorithm~\ref{fig:Ell_AS-DS}.

\begin{lemma}\label{feasD-AS}
 At every iteration $k$ of Algorithm~\texttt{EllAS},
 Problem~\eqref{lagr_dual_red_ask} admits a feasible solution.  
\end{lemma}
\begin{proof}
  It suffices to show that the ellipsoidal constraint
  \begin{equation}\label{ellcons}
    (c + {A^{(k)}}^\top\lambda^{(k)})^\top Q^{-1} (c+{A^{(k)}}^\top\lambda^{(k)})\leq 1 
  \end{equation}
  is satisfied for each $k$.  For $k=1$, this is explicitely required
  for the input of Algorithm~\texttt{EllAS}.  Let $\lambda^{(k)}$ be
  computed from~$\lambda^{(k-1)}$ by moving along the direction
  $p^{(k)}$.  The feasibility of $\lambda^{(k)}$ with respect
  to~\eqref{ellcons} then follows from the definition of $p^{(k)}$ and
  from the convexity of the ellipsoid.
  \qed
\end{proof}

\begin{proposition}\label{prop:lambdakfeasible}
 At every iteration $k$ of Algorithm~\texttt{EllAS},
 the vector $\lambda^{(k)}$ is feasible for~\eqref{lagr_dual_red}.  
\end{proposition}
\begin{proof}
Taking into account the proof of Lemma \ref{feasD-AS}, it remains to show nonnegativity of $\lambda^{(k)}$,
which is guaranteed by the choice of the steplength $\alpha^{(k)}$.
  \qed
\end{proof}

\begin{proposition}\label{prop:ascent}
  Assume that the steplength~$\alpha^k$ is always non-zero in the dual step.
  If Algorithm~\texttt{EllAS} does not stop at iteration $k$, then
  one of the following holds:
\begin{itemize}
 \item[(i)] $-{b^{(k+1)}}^\top \lambda^{(k+1)}> -{b^{(k)}}^\top \lambda^{(k)}$;
 \item[(ii)] $-{b^{(k+1)}}^\top \lambda^{(k+1)}= -{b^{(k)}}^\top \lambda^{(k)}$ and $\|\lambda^{(k+1)}\| <\|\lambda^{(k)}\|$.
\end{itemize}
\end{proposition}
\begin{proof}
  In the primal step, suppose that $\tilde \lambda^{(k)}\geq 0$ solves
  Problem~\eqref{lagr_dual_red_ask} and that the corresponding unique
  primal solution satisfies~$x^{(k)}~\not\in~P$.  After adding a
  violated cutting plane, the optimal value of Problem~\eqref{P-AS}
  strictly increases and the same is true for the optimal value of
  Problem~\eqref{D-AS} by strong duality.  Then, \[p^{(k+1)} = \tilde
  \lambda^{(k+1)} - \lambda^{(k)} = \tilde \lambda^{(k+1)} - \tilde
  \lambda^{(k)}\] is a strict ascent direction for $-b^\top \lambda$
  and case (i) holds.

  In the dual step, if $p^{(k+1)}$ is an unbounded direction, case (i)
  holds again.  Otherwise, observe that~$\lambda^{(k)}\neq
  \tilde\lambda^{(k+1)}$, as $\tilde\lambda^{(k+1)}$ is not feasible
  with respect to the nonnegativity constraints.  Then,
  since~$\tilde\lambda^{(k+1)}$ is the unique optimal solution for
  Problem~\eqref{lagr_dual_red_ask} with minimal norm, $p^{(k+1)} =
  \tilde \lambda^{(k+1)} - \lambda^{(k)}$ is either a strict ascent
  direction for $-b^\top \lambda$, or $-b^\top p^{(k+1)} = 0$ and
  $p^{(k+1)}$ is a strict descent direction for~$\|\lambda\|$, so that case
  (ii) holds.
  \qed
\end{proof}

\begin{lemma}\label{lemma:n+1steps}
 At every iteration $k$ of Algorithm~\texttt{EllAS}, we have $m^{(k)}\leq n+1$. Furthermore, 
 if Algorithm~\texttt{EllAS} terminates at iteration $k$ with an optimal primal-dual pair, 
 then $m^{(k)}\leq n$.
\end{lemma}
\begin{proof}
  As only violated cuts are added, the primal
  constraints~$A^{(k)}x=b^{(k)}$ either form an infeasible system or are linearly
  independent. If $m^{(k)} = n+1$, the primal problem is hence
  infeasible. Thus Problem~\eqref{lagr_dual_red_ask} is unbounded, so
  that at iteration $k$ a dual step is performed and a dependent
  row of~$(A^{(k)}, b^{(k)})$ is deleted, leading to an independent
  set of constraints again.
  \qed
\end{proof}

\begin{theorem}
  Assume that whenever a dual step is performed, Algorithm~\texttt{EllAS} takes a non-zero steplength~$\alpha^k$.
  Then, after at most $n 2^{m}$ iterations, Algorithm~\texttt{EllAS}
  terminates with a primal-dual pair of optimal solutions
  for~\eqref{ellLP} and~\eqref{lagr_dual_red}. 
\end{theorem}
\begin{proof}
First note that, by Lemma~\ref{lemma:n+1steps}, at most $n$ dual steps can be performed in a row.
Hence, it is enough to show that in any two iterations $k\neq l$ where a primal step is performed, 
we have $(A^{(k)}, b^{(k)})\neq (A^{(l)}, b^{(l)})$.
Otherwise, assuming $(A^{(k)}, b^{(k)}) = (A^{(l)}, b^{(l)})$, we obtain $\tilde\lambda^{(k)}=\tilde\lambda^{(l)}$ and hence
$\lambda^{(k)}=\lambda^{(l)}$. This leads to a contradiction to Proposition~\ref{prop:ascent}.
  \qed
\end{proof}

\subsection{Running time per iteration}
\label{sec:runningtime}

The running time in iteration~$k$ of \texttt{EllAS} is~$O(m^{(k)}n)$
and hence linear in the size of the matrix~$A^{(k)}$, if implemented
properly. The main work is to keep the
pseudo-inverse~$(A^{(k)}Q^{-\frac 12})^+$
up-to-date. Since~$A^{(k)}Q^{-\frac 12}$ is only extended or shrunk by
one row in each iteration, an update of~$(A^{(k)}Q^{-\frac 12})^+$ is possible
in~$O(m^{(k)}n)$ time by a generalization of the
Sherman-Morrison-formula~\cite{meyer73}. Exploiting the fact that the
matrix~$A^{(k)}$ has full row rank in most iterations, we can proceed as
follows. If~$A^{(k+1)}$ is obtained from~$A^{(k)}$ by adding a new
row~$a$, we first compute the row vectors
$$h:=aQ^{-\frac 12}(A^{(k)}Q^{-\frac 12})^+,\quad v:=aQ^{-\frac 12}-hA^{(k)}Q^{-\frac 12}\;.$$
Now~$v\neq 0$ if and only if~$A^{(k+1)}$ has full row rank, and in the latter case
$$(A^{(k+1)}Q^{-\frac 1 2})^+=\left((A^{(k)}Q^{-\frac 1 2})^+ \mid
0\right)-\tfrac 1{||v||^2}v^\top(h \mid -1)\;.$$ Otherwise, if~$v=0$,
we are adding a linearly dependent row to~$A^{(k)}$, making the primal
problem~\eqref{P-AS} infeasible. In this case, an unbounded direction
of steepest ascent of~\eqref{D-AS} is given by~$(-h \mid 1)^\top$ and
the next step will be a dual step, meaning that a row will be removed
from~$A^{(k+1)}$ and the resulting matrix~$A^{(k+2)}$ will have full row rank
again. We can thus update~$(A^{(k)}Q^{-\frac 1 2})^+$
to~$(A^{(k+2)}Q^{-\frac 1 2})^+$ by first removing and then adding a
row, in both cases having full row rank.

It thus remains to deal with the case of deleting the $r$-th row~$a$
of a matrix~$A^{(k)}$ with full row rank. Here we obtain
$(A^{(k+1)}Q^{-\frac 1 2})^+$ by deleting the $r$-th column in
$$(A^{(k)}Q^{-\frac 1 2})^+-\tfrac 1{||w||^2} ww^\top (A^{(k)}Q^{-\frac 1 2})^+\;,$$
where $w$ is the $r$-th column of $(A^{(k)}Q^{-\frac 1 2})^+$.

\begin{theorem}
  The running time per iteration of Algorithm~\texttt{EllAS} is~$O(n^2)$.
\end{theorem}
\begin{proof}
  This follows directly from Lemma~\ref{lemma:n+1steps} and the
  discussion above.
  \qed
\end{proof}

Clearly, the incremental update of the
pseudo-inverse~$(A^{(k)}Q^{-\frac 1 2})^+$ may cause numerical
errors. This can be avoided by recomputing it from scratch after a
certain number of incremental updates. Instead of a fixed number of
iterations, we recompute $(A^{(k)}Q^{-\frac 1 2})^+$ whenever the
primal solution computed in a primal step is infeasible, i.e.,
violates the current constraints, where we allow a small tolerance.

In order to avoid wrong solutions even when pseudo-inverses are not
precise, we make sure in our implementation that the dual
solution~$\lambda^{(k)}$ remains feasible for~\eqref{D-AS} in each
iteration, no matter how big the error of $(A^{(k)}Q^{-\frac 1 2})^+$
is. For this, we slightly change the computation of
$\tilde\lambda^{(k)}$: after computing~$\tilde\lambda^{(k)}$ exactly
as explained, we determine the largest~$\delta\in\R$ such
that~$(1-\delta)\lambda^{(k-1)}+\delta \tilde\lambda^{(k)}$ is dual
feasible. Such~$\delta$ must exist since~$\lambda^{(k-1)}$ is dual
feasible, and it can easily be computed using the midnight formula. We
then replace~$\tilde\lambda^{(k)}$ by
$(1-\delta)\lambda^{(k-1)}+\delta \tilde\lambda^{(k)}$ and go on as
before.

\section{Branch-and-Bound Algorithm}\label{sec:bb}

For solving
the integer Problem~\eqref{RP}, the method
presented in the previous sections must be embedded into a
branch-and-bound scheme. The dual bounds are computed by
Algorithm~\texttt{EllAS} and the branching is done by splitting up
the domain~$[l_i,u_i]$ of some variable~$x_i$. Several properties of
Algorithm~\texttt{EllAS} can be exploited to improve the performance
of such a branch-and-bound approach.

\paragraph{Warm starts}

Clearly, as branching adds new constraints to the primal feasible
region of the problem, while never extending it, all dual solutions
remain feasible. In every node of the branch-and-bound-tree, the
active set algorithm can thus be warm started with the optimal set of
constraints of the parent node. As 
in~\cite{buchheim:2016,buchheim:2018}, this leads to a significant reduction
of the number of iterations compared to a cold start. 
Moreover, the newly introduced bound constraint is always violated and can be
directly added as a new active constraint, which avoids resolving the
same dual problem and hence saves one more iteration per
node. Finally, the data describing the problem can either be inherited
without changes or updated quickly; this is particularly important for
the pseudo-inverse~$(AQ^{-\frac 12})^+$.

\paragraph{Early pruning}

Since we compute a valid dual bound for Problem~\eqref{RP} in every iteration of
Algorithm~\texttt{EllAS}, we may prune a subproblem as soon as the
current bound exceeds the value of the best known feasible solution.

\paragraph{Avoiding cycling or tailing off}

Last but not least, we may also stop Algorithm~\texttt{EllAS} at every point
without compromising the correctness of the branch-and-bound algorithm. In particular,
we can stop as soon as an iteration of Algorithm~\texttt{EllAS} does not give a
strict (or a significant) improvement in the dual bound. In
particular, this avoids cycling.


\section{Numerical Results}\label{sec:num}

To test the performance of our algorithm~\texttt{EllAS}, we considered
random binary instances with up to one million constraints
(Section~\ref{sec:num-rand}) as well as combinatorial instances of
Problem~\eqref{RP}, where the underlying problem is the Shortest Path
problem (Section~\ref{sec:num-sp}), the Assignment problem
(Section~\ref{sec:num-ap}), the Spanning Tree problem
(Section~\ref{sec:num-stp}), and the Traveling Salesman problem
(Section~\ref{sec:num-tsp}).  Concerning our approach, these
combinatorial problems have different characteristics: while the first
two problems have compact and complete linear formulations, the
standard models for the latter problems use an exponential number of
constraints that can be separated efficiently. In the case of the
Spanning Tree problem, this exponential set of constraints again
yields a complete linear formulation, while this is not the case for
the NP-hard Traveling Salesman problem. In the latter case, however,
we still have a complete integer programming formulation, which
suffices for the correctness of our approach.

For all problems, we consider instances where the positive definite matrix $Q\in \R^{n\times
  n}$ is randomly generated. For this, we chose $n$ eigenvalues $\lambda_i$ uniformly at random
from $[0,1]$ and orthonormalized $n$ random vectors $v_i$, each entry
of which was chosen uniformly at random from $[-1,1]$, then we set
$Q=\sum_{i=1}^n\lambda_iv_iv_i^{\top}$. For the random binary
instances, the entries of~$c$ were chosen uniformly at random from
$[-1,1]$, while for all remaining instances the vector~$c$ was
uniformly one.

In the following, we present a comparison of \texttt{BB-EllAS}, a~C++
implementation of the branch-and-bound-algorithm based
on~\texttt{EllAS}, with the MISOCP solver of Gurobi~7.5.1~\cite{gurobi}. According
to the latest benchmark results of Hans~D. Mittelmann~\cite{mittelmann}, Gurobi
is currently the fastest solver for MISOCPs. We use Gurobi with
standard settings, except that we use the same optimality tolerance as
in \texttt{BB-EllAS}, setting the absolute optimality tolerance
\texttt{MIPGapAbs} to $10^{-4}$.  All other standard parameters are
unchanged. In particular, Gurobi uses presolve techniques that
decrease the solution times significantly.  In case of the Spanning
Tree problem and the Traveling Salesman problem, we apply dynamic
separation algorithms using a callback adding lazy constraints.

All our experiments were carried out on Intel Xeon processors running
at 2.60~GHz. All running times were measured in CPU seconds and the
time-limit was set to one CPU hour for each individual instance. All
tables presented in this section include the following data for the
comparison between~\texttt{BB-EllAS} and Gurobi: the number of
instances solved within the time limit, the average running time, and
the average number of branch-and-bound nodes. For~\texttt{BB-EllAS}, we
also report the average total number of active set iterations and the
average number of times the pseudo-inverse $({A^{(k)}{Q^{-\frac 1
      2}}})^+$ is recomputed from scratch, the latter in percentage with respect
to the number of iterations. All averages are taken over the set of
instances solved within the time limit.
For all applications, we also
present performance profiles, as proposed in~\cite{DM2002}.
Given our set of solvers $\mathcal{S}$=\{\texttt{BB-EllAS}, Gurobi\} and a set of problems $\mathcal{P}$, 
we compare the performance of  a solver $s \in \mathcal{S}$ on problem $p \in \mathcal{P}$ 
against the best performance obtained by any solver in $\mathcal{S}$ 
on the same problem. To this end we define the performance ratio
$
r_{p,s} = t_{p,s}/\min\{t_{p,s^\prime}: s^\prime \in\mathcal{S}\},
$
where $t_{p,s}$ is the computational time, and we consider a cumulative distribution
function
\[
\rho_s(\tau) = |\{p\in \mathcal{P}:\; r_{p,s}\leq \tau \}| /|\mathcal{P}|.
\]
The performance profile for $s \in S$ is the plot of the function $\rho_s$.
\subsection{Random Instances}\label{sec:num-rand}

For a first comparison, we consider instances of
Problem~\eqref{ellvILP} where the objective function vector $c\in
\R^n$ and the positive definite matrix $Q\in \R^{n\times n}$ are
randomly generated as described above. The set $P$ is explicitely
given as~$\{x\in \R^n \mid Ax \leq b\}$, where $A\in \R^{m\times n}$
and $b\in \R^m$ are also randomly generated: the entries of~$A$ were
chosen uniformly at random from the integers in the range~$[0,10]$ and $b$ was defined by
$b_i=\lfloor\frac 12{\sum_{j=1}^na_{ij}}\rfloor$, $i=1,\dots,m$.  Altogether, we
generated 160 different problem instances for~\eqref{ellvILP}: for
each combination of $n\in\{25,50,100,200\}$ and
$m\in\{10^3,10^4,10^5,10^6\}$, we generated 10 instances.  Since the
set $P$ is explicitely given here, the linear constraints are
separated by enumeration in~\texttt{BB-EllAS}.  More precisely, at
Step~\ref{Step:separation} of Algorithm~\ref{fig:Ell_AS-PS}, we pick
the linear constraint most violated by $x^{(k)}$.  We report our
results in Table~\ref{Tab:rnd01}.
\begin{table}
  \caption{Comparison on random binary instances.}\label{Tab:rnd01} 
  \centering 
  \begin{tabular}{|rr| r r r r r| r r r|} 
\hline
 & & \multicolumn{5}{c|}{BB-EllAS} & \multicolumn{3}{c|}{Gurobi} \\
  $n$ & $m$ & \#sol & time & nodes & iter & \%ps & \#sol & time & nodes \\
\hline
  25 & $10^3$ & 10 &     0.00 &  3.9e+01 &  2.5e+02 &  0.12 & 10 &     0.76 &  1.3e+01 \\
  25 & $10^4$ & 10 &     0.03 &  6.5e+01 &  4.9e+02 &  0.45 & 10 &     9.39 &  2.1e+01 \\
  25 & $10^5$ & 10 &     0.60 &  1.0e+02 &  9.7e+02 &  1.22 & 10 &   156.63 &  2.8e+01 \\
  25 & $10^6$ & 10 &    16.91 &  2.5e+02 &  2.5e+03 &  0.85 & 10 &  1973.87 &  8.6e+01 \\
\hline
  50 & $10^3$ & 10 &     0.01 &  6.3e+01 &  3.5e+02 &  0.54 & 10 &     0.96 &  1.1e+01 \\
  50 & $10^4$ & 10 &     0.05 &  6.7e+01 &  4.4e+02 &  0.62 & 10 &    11.93 &  1.8e+01 \\
  50 & $10^5$ & 10 &     0.85 &  7.8e+01 &  6.8e+02 &  1.13 & 10 &   246.32 &  2.1e+01 \\
  50 & $10^6$ & 10 &    18.84 &  1.7e+02 &  1.7e+03 &  1.22 &  0 &      --- &      --- \\
\hline
 100 & $10^3$ & 10 &     0.40 &  1.3e+02 &  8.3e+02 &  6.79 & 10 &     2.35 &  2.5e+01 \\
 100 & $10^4$ &  9 &     0.36 &  1.6e+02 &  1.1e+03 &  1.78 & 10 &    27.51 &  7.7e+01 \\
 100 & $10^5$ &  9 &     4.26 &  2.6e+02 &  2.0e+03 &  1.60 & 10 &   761.07 &  2.3e+02 \\
 100 & $10^6$ &  7 &    94.88 &  4.9e+02 &  4.8e+03 &  3.68 &  0 &      --- &      --- \\
\hline
 200 & $10^3$ & 10 &     1.00 &  1.3e+02 &  7.6e+02 &  3.98 & 10 &     4.55 &  3.1e+01 \\
 200 & $10^4$ &  8 &     2.04 &  1.6e+02 &  1.3e+03 &  3.92 & 10 &    23.63 &  4.1e+01 \\
 200 & $10^5$ &  9 &    11.84 &  3.1e+02 &  2.6e+03 &  3.15 & 10 &   899.84 &  1.3e+02 \\
 200 & $10^6$ &  7 &    61.25 &  1.9e+02 &  1.4e+03 & 14.11 &  0 &      --- &      --- \\
\hline
  \end{tabular}
\end{table}

From the results in Table~\ref{Tab:rnd01}, note that the average
number of branch-and-bound nodes enumerated by~\texttt{BB-EllAS} is
generally larger than the number of nodes needed by Gurobi, but always
by less than a factor of~10 on average. However, in terms of running
times,~\texttt{BB-EllAS} outperforms Gurobi on all instance types
except for the larger instances with a medium number of constraints,
i.e., for~$n\in\{100,200\}$ and $m\in\{10^4,10^5\}$. On all other
instance classes,~\texttt{BB-EllAS} either solves significantly more
instances than Gurobi within the time limit or has a faster running
time by many orders of magnitude.  This in confirmed by the
performance profiles presented in Figure~\ref{fig:perfprof_rand}. The low
number of iterations performed by~\texttt{EllAS} per node (less than
10 on average) highlights the benefits of using warmstarts.

\begin{figure}[h!]
  \begin{center}
    \psfrag{GUROBI}[lc][lc]{\scriptsize\texttt{Gurobi}}
    \psfrag{BB-EllAS}[lc][lc]{\scriptsize\texttt{BB-EllAS}}
    \includegraphics[trim = 1.5cm 0cm 1.5cm 0mm, clip, width=0.85\textwidth]{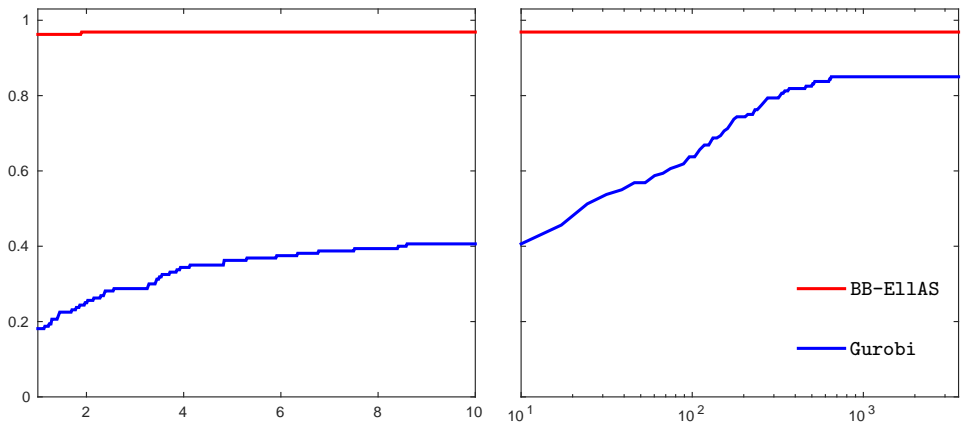}\\
    \caption{Performance profile with respect to running times for
      random binary instances.}
    \label{fig:perfprof_rand}
  \end{center}
\end{figure}



\subsection{Shortest Path Problem}\label{sec:num-sp}

Given a directed graph $G =(V,E)$, where
$V$ is the set of vertices
and $E$ is the set of edges, and weights
associated with each edge, the Shortest Path problem 
is the problem of finding a path between two vertices~$s$ and~$t$ such 
that the sum of the weights of its constituent edges is minimized.
Our approach uses the following flow based formulation of the Robust
Shortest Path problem:
\begin{equation}\label{sp}
  \begin{array}{lrcll}
    \min~ & \multicolumn{4}{l}{c^\top x+\sqrt{x^ \top Qx}}\\[1.5ex]
    \textnormal{ s.t.} & \sum_{e  \in \delta^+ (i)}x_e - \sum_{e  \in \delta^- (i)}x_e & = & 0 & \forall i \in V \setminus \{s,t\} \\[1ex]
    & \sum_{e  \in \delta^+ (s)} x_e - \sum_{e  \in \delta^- (s)} x_e & = & 1 & \\[1ex]
    & \sum_{e  \in \delta^+ (t)} x_e - \sum_{e  \in \delta^- (t)} x_e & = & -1 & \\[1ex]
    & x & \in & \{0,1\}^{E}
  \end{array}
\end{equation}
In our test set, we
produced squared grid graphs with $r$ rows and columns, where all
edges point from left to right and from top to bottom. In this way, we
produced graphs with $|V| = r^2$ vertices and $|E| = 2 r^2 -2r$
edges. In the IP model~\eqref{sp}, we thus have $n:=2 r^2 -2r$ variables
and $m:=|E|+2|V|=4 r^2 -2r$ many inequalities, taking into account
also the box constraints $x_e\in[0,1]$ for $e\in E$. Since this number
is polynomial in~$n$, we can separate them by enumeration
within~\texttt{EllAS}, whereas we can pass the formulation~\eqref{sp} to Gurobi directly.
Concerning the objective function of~\eqref{sp}, we set all expected
lenghts~$c_i$ to~$1$ and built the positive definite matrix~$Q$ as described above.
Altogether, we generated 100 different problem instances for~\eqref{sp}:
for each $r\in\{5,\ldots,14\}$ we generated 10 instances.

\begin{table}
  \caption{Comparison on robust shortest path instances.}\label{Tab:sp} 
  \centering 
  \begin{tabular}{|rrr| r r r r r| r r r|} 
\hline
 &&& \multicolumn{5}{c|}{BB-EllAS} & \multicolumn{3}{c|}{Gurobi} \\
  $r$ & $n$ & $m$ & \#sol & time & nodes & iter & \%ps & \#sol & time & nodes \\
\hline
  5  &  40 &   90 & 10 &     0.01 &  4.4e+01 &  4.9e+02 &  0.00 & 10 &     0.17 &  5.0e+01  \\   
  6  &  60 &  132 & 10 &     0.04 &  1.2e+02 &  1.5e+03 &  0.45 & 10 &     0.65 &  1.4e+02  \\   
  7  &  84 &  182 & 10 &     1.26 &  3.8e+02 &  5.4e+03 &  5.51 & 10 &     2.29 &  3.5e+02  \\   
  8  & 112 &  240 & 10 &     2.52 &  9.7e+02 &  1.6e+04 &  1.68 & 10 &     8.86 &  8.9e+02  \\   
  9  & 144 &  306 & 10 &     5.71 &  2.5e+03 &  4.7e+04 &  0.42 & 10 &    92.12 &  3.0e+03  \\   
  10 & 180 &  380 & 10 &    67.68 &  6.0e+03 &  1.3e+05 &  1.99 & 10 &   349.19 &  6.6e+03  \\   
  11 & 220 &  462 & 10 &   214.08 &  1.9e+04 &  4.4e+05 &  0.93 & 10 &  1294.86 &  1.8e+04  \\   
  12 & 264 &  552 & 10 &   659.91 &  4.8e+04 &  1.3e+06 &  0.47 &  1 &  1682.57 &  2.1e+04  \\   
  13 & 312 &  650 &  3 &  2925.07 &  1.2e+05 &  3.4e+06 &  0.51 &  0 &      --- &      ---  \\   
  14 & 364 &  756 &  0 &      --- &      --- &      --- &   --- &  0 &      --- &      ---  \\   
\hline
  \end{tabular}
\end{table}

In Table~\ref{Tab:sp}, we report the comparison between~\texttt{BB-EllAS} and the MISOCP solver of~Gurobi.
The average number of branch-and-bound nodes needed in~\texttt{BB-EllAS} is 
in the same order of magnitude of that needed by~Gurobi.
However,~\texttt{EllAS} is able to process the nodes very quickly, leading to a branch-and-bound 
scheme that outperforms~Gurobi in terms of computational time. 
Note also that for graphs with $r=13$,~Gurobi does not solve any instance within one hour CPU time, 
while~\texttt{BB-EllAS} is able to solve $3$ of them.
Both solvers fail for instances with $r\ge 14$. See
Figure~\ref{fig:perfprof_sp} for the performance profiles.

\begin{figure}[h!]
  \begin{center}
    \psfrag{GUROBI}[lc][lc]{\scriptsize\texttt{Gurobi}}
    \psfrag{BB-EllAS}[lc][lc]{\scriptsize\texttt{BB-EllAS}}
    \includegraphics[trim = 1.5cm 0cm 1.5cm 0mm, clip, width=0.85\textwidth]{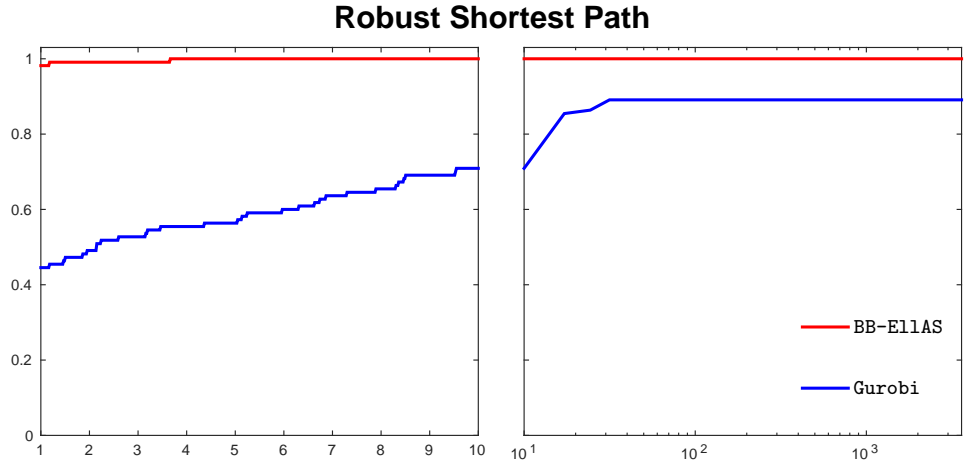}\\
    \caption{Performance profile with respect to running times for shortest path instances.}
    \label{fig:perfprof_sp}
  \end{center}
\end{figure}

\subsection{Assigment Problem}\label{sec:num-ap}

Given an undirected, bipartite and weighted graph $G =(V,E)$ with
bipartition~$V=V_1\cup V_2$, the Assignment problem consists in
finding a one-to-one assignment from the nodes in~$V_1$ to the nodes
in~$V_2$ such that the sum of the weights of the edges used for the assignment
is minimized. In other words, we search for a minimum-weight perfect
matching in the bipartite graph~$G$.  Our approach uses the following
standard formulation of the Assignment problem:
\begin{equation*}
  \begin{array}{lrcll}
    \min~ & \multicolumn{4}{l}{c^\top x+\sqrt{x^ \top Qx}}\\[1.5ex]
    \textnormal{ s.t.} & \sum_{e\in \delta(i)}x_e & = & 1 & \forall i\in V \\[1ex]
    & x & \in & \{0,1\}^{E}
  \end{array}
\end{equation*}
We consider complete bipartite graphs, so that the number of variables
is $n=\tfrac 14 |V|^2$. The number of constraints including $x\ge 0$
is $m=|V|+n$. Note that in the bipartite case the above formulation
yields a complete description of~$\text{conv}(P\cap \Z^n)$, which is not true in the
case of general graphs. In our instances, we use expected weights~$1$
again, while the non-linear part of the objective function is
generated as before.  Altogether, we generated
80 different problem instances: for each $|V|\in\{10,12,\ldots,24\}$ we
generated 10 different instances.
Results are presented in Table~\ref{tab:ap} and Figure~\ref{fig:perfprof_ap}. The general picture is
very similar to the one for the Shortest Path problem.

\begin{table}
  \caption{Comparison on robust assignment instances.}\label{tab:ap}
  \centering 
  \begin{tabular}{|rrr| r r r r r| r r r|} 
\hline
 &&& \multicolumn{5}{c|}{BB-EllAS} & \multicolumn{3}{c|}{Gurobi} \\
  $|V|$ & $n$ & $m$ & \#sol & time & nodes & iter & \%ps & \#sol & time & nodes \\
\hline
  10 &  25 &  35 & 10 &     0.00 &  7.9e+01 &  6.2e+02 &  0.07 & 10 &     0.11 &  8.3e+01 \\
  12 &  36 &  48 & 10 &     0.01 &  2.6e+02 &  2.3e+03 &  0.07 & 10 &     0.48 &  2.8e+02 \\
  14 &  49 &  63 & 10 &     0.11 &  1.1e+03 &  9.9e+03 &  0.09 & 10 &     2.83 &  1.1e+03 \\
  16 &  64 &  80 & 10 &     1.34 &  4.5e+03 &  4.8e+04 &  0.73 & 10 &    29.98 &  4.4e+03 \\
  18 &  81 &  99 & 10 &    11.90 &  2.6e+04 &  3.0e+05 &  0.49 & 10 &   198.47 &  2.0e+04 \\
  20 & 100 & 120 & 10 &   112.16 &  1.2e+05 &  1.5e+06 &  0.85 & 10 &  1521.62 &  1.0e+05 \\
  22 & 121 & 143 &  9 &   669.68 &  6.6e+05 &  8.9e+06 &  0.36 &  0 &      --- &      --- \\
  24 & 144 & 168 &  0 &      --- &      --- &      --- &   --- &  0 &      --- &      --- \\
\hline
  \end{tabular}
\end{table}

\begin{figure}[h!]
  \begin{center}
    \psfrag{GUROBI}[lc][lc]{\scriptsize\texttt{Gurobi}}
    \psfrag{BB-EllAS}[lc][lc]{\scriptsize\texttt{BB-EllAS}}
    \includegraphics[trim = 1.5cm 0cm 1.5cm 0mm, clip, width=0.85\textwidth]{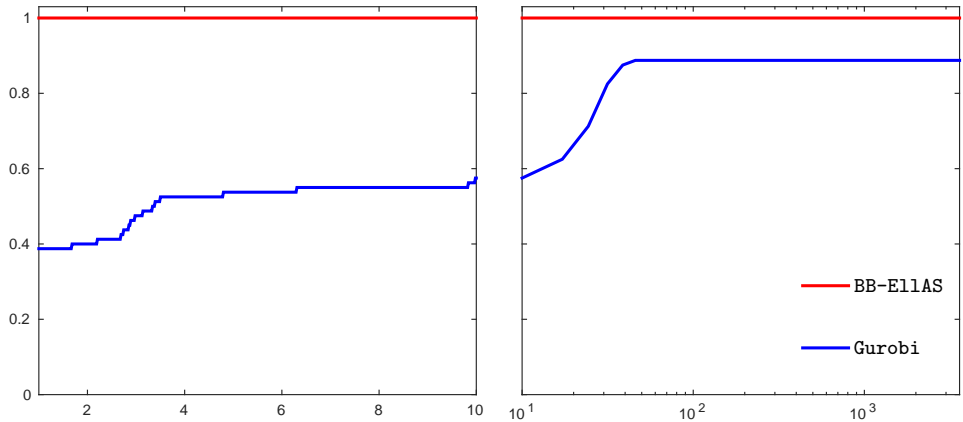}\\
    \caption{Performance profile with respect to running times for assignment instances.}
    \label{fig:perfprof_ap}
  \end{center}
\end{figure}

\subsection{Spanning Tree Problem}\label{sec:num-stp}

Given an undirected weighted graph
$G =(V,E)$, a minimum spanning tree is a subset of 
edges that connects all vertices, 
without any cycles and with the minimum total edge weight. 
Our approach uses the following formulation of the Robust Spanning Tree problem:
\begin{equation}\label{mst}
  \begin{array}{rrcll}
    \min~ & \multicolumn{4}{l}{c^\top x+\sqrt{x^ \top Qx}}\\[1.5ex]
    \textnormal{s.t.} & \sum_{e\in E}x_e & = & |V|-1  \\[1ex]
    & \sum_{e\subseteq X} x_e & \le & |X|-1 & \forall \emptyset\neq X\subseteq V\\[1ex]
    & x & \in & \{0,1\}^{E}
  \end{array}
\end{equation}
In the above model, the number of constraints,
taking into account also the non-negativity constraints, is
$m=2^{|V|}+n$. Since this number is exponential, we also have to use a
separation algorithm for Gurobi. For both \texttt{BB-EllAS}
and Gurobi, we essentially use the same simple implementation based
on the Ford-Fulkerson algorithm.

For our experiments, we considered both complete graphs and grid
graphs, the latter being produced as for the Shortest Path Problem. In
both cases, expected edge weights are set to~$1$ again, while we built
the positive definite matrix~$Q$ as above.  Altogether, we generated
90 different problem instances: for each $|V|\in\{10,\ldots,14\}$ we
generated 10 different complete instances, while for each
$r\in\{5,\ldots,8\}$ we generated 10 different grid instances. As shown
in Tables~\ref{tab:mst} and~\ref{tab:mstgrid},~\texttt{BB-EllAS}
clearly outperforms the MISOCP solver of~Gurobi on all the instances
considered. For the performance profiles, see
Figure~\ref{fig:perfprof_mst}.

\begin{table}
  \caption{Comparison on robust minimum spanning tree instances (complete graphs).}\label{tab:mst}
  \centering 
  \begin{tabular}{|rrr| r r r r r |r r r |} 
    \hline
 &&& \multicolumn{5}{c|}{BB-EllAS} & \multicolumn{3}{c|}{Gurobi} \\
  $|V|$ & $n$ & $m$ & \#sol & time & nodes & iter & \%ps & \#sol  & time & nodes \\
\hline
  10 &  45 &   1,069 & 10 &     2.93 &  1.4e+04 &  9.6e+04 &  1.35 & 10 &    78.59 &  1.6e+04 \\
  11 &  55 &   2,103 & 10 &    11.92 &  5.7e+04 &  4.3e+05 &  0.16 & 10 &   794.29 &  7.0e+04 \\
  12 &  66 &   4,162 & 10 &   120.84 &  4.4e+05 &  3.7e+06 &  0.06 &  1 &  2652.38 &  1.4e+05 \\
  13 &  78 &   8,270 & 10 &  1060.63 &  2.7e+06 &  2.4e+07 &  0.12 &  0 &      --- &      --- \\
  14 &  91 &  16,475 &  0 &      --- &      --- &      --- &   --- &  0 &      --- &      --- \\
\hline
  \end{tabular}
\end{table}

\begin{table}
  \caption{Comparison on robust minimum spanning tree instances (grid graphs).}\label{tab:mstgrid}
  \centering 
  \begin{tabular}{|rrr| r r r r r |r r r |} 
    \hline
 &&& \multicolumn{5}{c|}{BB-EllAS} & \multicolumn{3}{c|}{Gurobi} \\
  $r$ & $n$ & $m$ & \#sol & time & nodes & iter & \%ps & \#sol & time & nodes \\
\hline
  5  &  40 & 3.4e+07 & 10 &     0.50 &  1.0e+03 &  8.7e+03 &  0.10 & 10 &    61.35 &  7.3e+03  \\   
  6  &  60 & 6.9e+10 & 10 &    18.64 &  1.1e+04 &  1.2e+05 &  0.67 &  8 &  1805.72 &  9.2e+04  \\   
  7  &  84 & 5.6e+14 &  9 &  1038.24 &  2.3e+05 &  3.3e+06 &  0.38 &  0 &      --- &      ---  \\   
  8  & 112 & 1.8e+19 &  0 &      --- &      --- &      --- &   --- &  0 &      --- &      ---  \\   
\hline
  \end{tabular}
\end{table}

\begin{figure}[h!]
  \begin{center}
    \psfrag{GUROBI}[lc][lc]{\scriptsize\texttt{Gurobi}}
    \psfrag{BB-EllAS}[lc][lc]{\scriptsize\texttt{BB-EllAS}}
    \includegraphics[trim = 1.5cm 0cm 1.5cm 0mm, clip, width=0.85\textwidth]{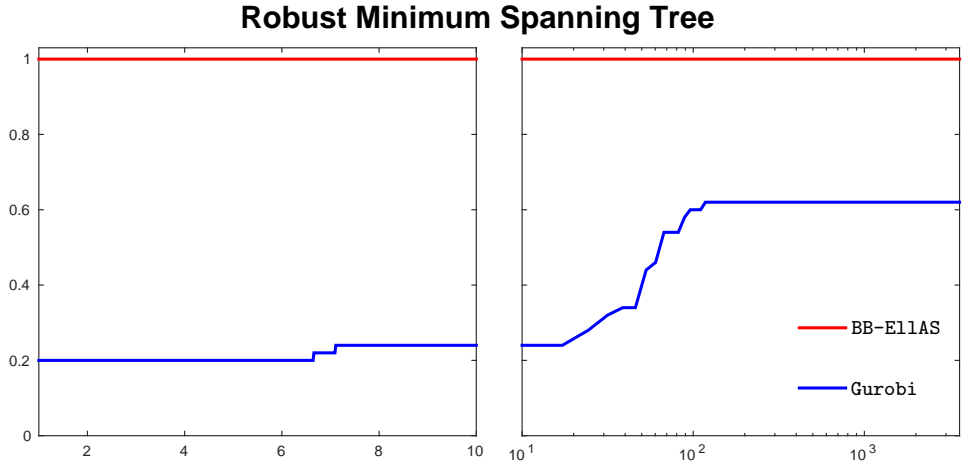}\\
    \caption{Performance profile with respect to running times for spanning tree instances.}
    \label{fig:perfprof_mst}
  \end{center}
\end{figure}

\subsection{Traveling Salesman Problem}\label{sec:num-tsp}

Given an undirected, complete and weighted graph $G =(V,E)$, the Traveling Salesman problem consists in
finding a path starting and ending at a given vertex~$v\in V$ such that all the vertices 
in the graph are visited exactly once and the sum of the weights of its constituent edges is minimized.
Our approach uses the following formulation of the Traveling Salesman problem:
\begin{equation}\label{tsp}
  \begin{array}{rrcll}
    \min~ & \multicolumn{4}{l}{c^\top x+\sqrt{x^ \top Qx}}\\[1.5ex]
    \textnormal{s.t.} & \sum_{e\in \delta(i)}x_e & = & 2 & \forall i\in V \\[1ex]
    & \sum_{e\in \delta(X)} x_e & \ge & 2 & \forall \emptyset\neq X\subsetneq V\\[1ex]
    & x & \in & \{0,1\}^{E}
  \end{array}
\end{equation}
Again, we consider complete graphs. The number of constraints
including the bounds $x\in[0,1]^E$ is $m=2^{|V|}+3n-2$ and hence again
exponential.  For both \texttt{BB-EllAS} and Gurobi, we basically use
the same separation algorithm as for the Spanning Tree problem; see
Section~\ref{sec:num-stp}. Instances are identical to those generated
for the Spanning Tree problem, but we can consider slightly larger
graphs, namely graphs with $|V|\in\{10,\ldots,16\}$. See
Table~\ref{tab:tsp} and Figure~\ref{fig:perfprof_tsp} for the results.

\begin{table}
  \caption{Comparison on robust traveling salesman instances.}\label{tab:tsp}
  \centering 
  \begin{tabular}{|rrr| r r r r r| r r r|} 
\hline
 &&& \multicolumn{5}{c|}{BB-EllAS} & \multicolumn{3}{c|}{Gurobi} \\
  $|V|$ & $n$ & $m$ & \#sol & time & nodes & iter & \%ps & \#sol & time & nodes \\
\hline
  10 &  45 &  1,157 & 10 &     0.70 &  3.5e+03 &  2.5e+04 &  1.73 & 10 &    15.50 &  3.3e+03 \\
  11 &  55 &  2,211 & 10 &     2.37 &  1.6e+04 &  1.3e+05 &  0.18 & 10 &    69.96 &  1.2e+04 \\
  12 &  66 &  4,292 & 10 &    17.59 &  9.6e+04 &  8.2e+05 &  0.05 & 10 &   637.47 &  7.1e+04 \\
  13 &  78 &  8,424 & 10 &   150.00 &  5.4e+05 &  4.8e+06 &  0.14 &  3 &  2324.42 &  2.3e+05 \\
  14 &  91 & 16,655 & 10 &  1087.76 &  2.5e+06 &  2.4e+07 &  0.25 &  0 &      --- &      --- \\
  15 & 105 & 33,081 &  1 &  2966.10 &  6.2e+06 &  6.0e+07 &  0.06 &  0 &      --- &      --- \\
  16 & 120 & 65,894 &  0 &      --- &      --- &      --- &   --- &  0 &      --- &      --- \\
\hline
  \end{tabular}
\end{table}

\begin{figure}[h!]
  \begin{center}
    \psfrag{GUROBI}[lc][lc]{\scriptsize\texttt{Gurobi}}
    \psfrag{BB-EllAS}[lc][lc]{\scriptsize\texttt{BB-EllAS}}
    \includegraphics[trim = 1.5cm 0cm 1.5cm 0mm, clip, width=0.85\textwidth]{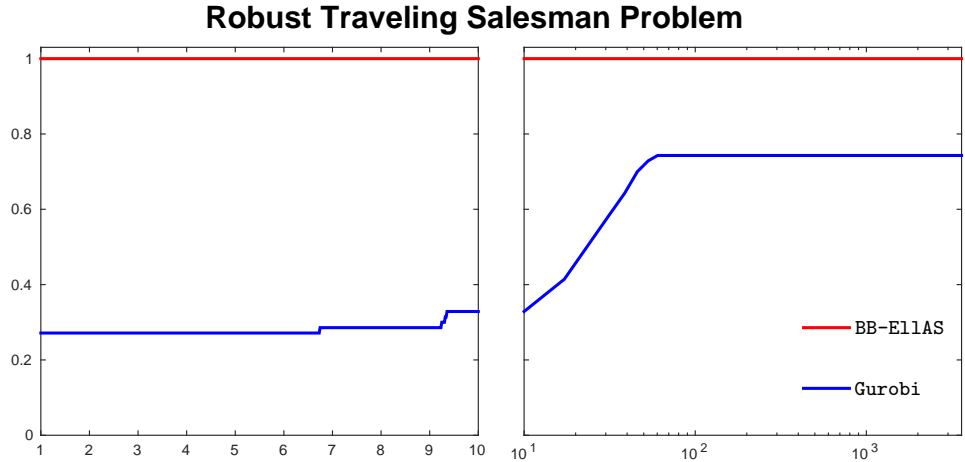}\\
    \caption{Performance profile with respect to running times for traveling salesman instances.}
    \label{fig:perfprof_tsp}
  \end{center}
\end{figure}

\section{Conclusions}\label{sec:conc}
We presented a new branch-and-bound algorithm for robust combinatorial optimization problems 
under ellipsoidal uncertainty. We assume that the set of feasible solutions is given by a separation
algorithm that decides whether a given point belongs to the convex
hull of the feasible set or not, and, in the negative case, provides
a valid but violated inequality. The branch-and-bound algorithm is based on the use of an active set method for the 
computation of dual bounds.  Dealing with the Lagrangian dual of the continuous relaxation 
has the advantage of allowing an early pruning of the node.
The closed form solution of the active set subproblems, the smart update of pseudo-inverse matrices,
as well as the possibility of using warmstarts, leads to an algorithm 
that clearly outperforms the mixed-integer SOCP solver of Gurobi on the problem instances considered,
including the robust counterpart of the shortest path and the traveling salesman problem.

\section{Acknowledgments}
 
The first author acknowledges support within the project 
``Mixed-Integer Non Linear Optimisation: Algorithms and Applications'', which has received 
funding from the Europeans Union’s EU Framework Programme for Research and Innovation 
Horizon 2020 under the Marie Skłodowska-Curie Actions Grant Agreement No 764759. 
The second author acknowledges support within the project ``Nonlinear Approaches for the Solution 
of Hard Optimization Problems with Integer Variables''(No RP11715C7D8537BA) which has received funding
from Sapienza, University of Rome.

\bibliographystyle{plain}
\bibliography{ellas}

\end{document}